\documentclass[11pt,french,english]{amsart}
\usepackage[T1]{fontenc}
\usepackage[latin1]{inputenc}
\usepackage{latexsym}
\usepackage{amsthm}
\usepackage{amssymb}

\numberwithin{equation}{section} 
\numberwithin{figure}{section} 
\theoremstyle{plain}
\theoremstyle{plain}
\newtheorem{thm}{Theorem}
  \theoremstyle{plain}
  \newtheorem{prop}[thm]{Proposition}
  \theoremstyle{remark}
  \newtheorem*{acknowledgement*}{Acknowledgement}
  \theoremstyle{plain}
  \newtheorem{lem}[thm]{Lemma}
 \theoremstyle{definition}
  \newtheorem{example}[thm]{Example}
  \theoremstyle{plain}
  \newtheorem{cor}[thm]{Corollary}
  \theoremstyle{remark}
  \newtheorem{rem}[thm]{Remark}
  \theoremstyle{remark}
  \newtheorem{notation}[thm]{Notation}

\AtBeginDocument{
  
}

\usepackage{babel}
\addto\extrasfrench{\providecommand{\fg}{\ifdim\lastskip>\z@\unskip\fi~\frqq}}

\begin{document}
\newcommand{\Alb}{{\rm Alb}}

\newcommand{\Jac}{{\rm Jac}}

\newcommand{\Hom}{{\rm Hom}}

\newcommand{\End}{{\rm End}}

\newcommand{\Aut}{{\rm Aut}}

\newcommand{\NS}{{\rm NS}}

\title{Fano surfaces with $12$ or $30$ elliptic curves}

\author{Xavier Roulleau}

\maketitle


\subsection*{Introduction.}

A Fano surface is a surface of general type that parametrizes the
lines of a smooth cubic threefold. First studied by Fano and then
by many others, like Bombieri and Swinnerton-Dyer \cite{Bombieri},
Gherardelli \cite{Gherardelli}, Tyurin \cite{Tyurin}, \cite{Tyurin1},
Clemens and Griffiths \cite{Clemens}, Collino \cite{Collino}, these
surfaces carry many remarkable properties. In our previous paper \cite{Roulleau1},
we classified Fano surfaces according to the configurations of their
elliptic curves. The aim of the present paper is to give various applications
of this study when the Fano surface contains $12$ or $30$ elliptic
curves.\\
The main result of the first part of this paper is as follows:
\begin{prop}
\label{Corollaire revetement de degre 3}The Picard number $\rho_{S}$
of a Fano surface $S$ satisfies $1\leq\rho_{S}\leq25$ and is $1$
for $S$ generic. \\
A Fano surface that contains $12$ elliptic curves is a triple
ramified cover of the blow-up of $9$ points of an abelian surface.\\
 The Néron-Severi group of such a surface has rank $12,\,13$ or
$25=h^{1,1}(S)$.\\
 For $S$ generic among Fano surfaces with $12$ elliptic curves,
the Néron-Severi group has rank $12$ and is rationally generated
by its $12$ elliptic curves.\\
An infinite number of Fano surfaces with $12$ elliptic curves
have maximal Picard number $25=h^{1,1}(S)$.
\end{prop}
Recall that among the K3 surfaces, the Kummer surfaces are recognized
as those K3 having $16$ disjoint $(-2)$-curves. They are the double
cover of the blow-up over the $2$-torsion points of an abelian surface
(see \cite{Nikulin}). Our theorem is the analogue for Fano surfaces
that contains 12 elliptic curves among Fano surfaces.

In the second part, we study the Fano surface $S$ of the Fermat cubic
threefold $F\hookrightarrow\mathbb{P}^{4}$:\[
F=\{x_{1}^{3}+x_{2}^{3}+x_{3}^{3}+x_{4}^{3}+x_{5}^{3}=0\}.\]
Let $\mu_{3}$ be the group of third roots of unity and let $\alpha\in\mu_{3}$
be a primitive root. For $s$ a point of $S$, we denote by $L_{s}$
the line on $F$ corresponding to the point $s$ and we denote by
$C_{s}$ the incidence divisor that parametrizes the lines in $F$
that cut the line $L_{s}$. 
\begin{thm}
The surface $S$ is the unique Fano surface that contains $30$ smooth
curves of genus $1$. These curves are numbered:\[
E_{ij}^{\beta},\,1\leq i<j\leq5,\,\beta\in\mu_{3},\]
 in such a way that for two such curves $E_{ij}^{\gamma}$ and $E_{st}^{\beta}$,
we have: \[
E_{ij}^{\beta}E_{st}^{\gamma}=\left\{ \begin{array}{cc}
1 & \textrm{if }\{i,j\}\cap\{s,t\}=\emptyset\\
-3 & \textrm{if }E_{ij}^{\beta}=E_{st}^{\gamma}\\
0 & \textrm{else}.\end{array}\right.\]
The Néron-Severi group $\NS(S)$ of $S$ has rank : $25=\dim H^{1}(S,\Omega_{S})$
and discriminant $3^{18}$. These $30$ elliptic curves generate an
index $3$ sub-lattice of $\NS(S)$ and with the class of an incidence
divisor $C_{s}$ ($s\in S$), they generate the Néron-Severi group.
\end{thm}
Given a smooth curve of low genus and with a sufficiently large automorphism
group, it is sometimes possible to calculate the period matrix of
its Jacobian \cite{Birkenhake}. In this paper, we calculate also
the period lattice of the Albanese variety of the $2$ dimensional
variety $S$. This computation is used to determine the Néron-Severi
group of $S$. We determine also the fibrations of $S$ onto an elliptic
curve and the intersection numbers between the fibers of these fibrations
and discuss on the more interesting fibrations. 
\begin{acknowledgement*}
A part of this paper done during the author's stay at the Max Planck
Institute of Bonn and at the University of Tokyo, under JSPS fellowship. 
\end{acknowledgement*}

\section{Preliminaries on Fano surfaces.}

\subsection{Tangent Bundle Theorem.\label{parag. proprietes des SdF.}}

Let $F\hookrightarrow\mathbb{P}^{4}$ be a smooth cubic threefold
and let $S$ be its Fano surface of lines. We consider the following
diagram : \[
\begin{array}{ccc}
\mathcal{U} & \stackrel{\psi}{\rightarrow} & F\hookrightarrow\mathbb{P}^{4}\\
\pi\downarrow\\
S\end{array}\]
where $\mathcal{U}$ is the universal family of lines and $\pi,\psi$
are the projections. 
\begin{thm}
(Tangent Bundle Theorem \cite{Clemens}). There is an isomorphism
: $\pi_{*}\psi^{*}\mathcal{O}(1)\simeq\Omega_{S}$, where $\Omega_{S}$
is the cotangent sheaf. \\
By this isomorphism, we can identify the spaces $H^{0}(F,\mathcal{O}(1))$
and $H^{0}(S,\Omega_{S})$, the varieties $\mathbb{P}^{4}$ and $\mathbb{P}(H^{0}(S,\Omega_{S})^{*})$
and the varieties $\mathcal{U}$ and $\mathbb{P}(T_{S})$, where $T_{S}=\Omega_{S}^{*}$. 
\end{thm}
We always work with the identifications of the above Theorem. In particular,
the line $L_{s}\hookrightarrow F$ corresponding to a point $s$ in
$S$ is the projectivized tangent space to $S$ at $s$.

\subsection{Properties of Fano surfaces with an elliptic curve.}

Let us denote by $\mathcal{C}$ a cone on the cubic $F$. The following
lemmas come from \cite{Roulleau1}:
\begin{lem}
\label{lem:The-base-of}The cone $\mathcal{C}$ is an hyperplane section
of $F$. The curve $E$ parametrizing the lines on $\mathcal{C}$
is naturally embedded into $S$ and is an elliptic curve. \\
Conversely, if $E\hookrightarrow S$ is an elliptic curve on $S$,
the surface $\psi(\pi^{-1}(E))$ is a cone. 
\end{lem}
Let $E\hookrightarrow S$ be an elliptic curve. 
\begin{lem}
\label{lem:construction de la fibration}We can canonically associate
to $E\hookrightarrow S$ an involution $\sigma_{E}:S\rightarrow S$
and a fibration $\gamma_{E}:S\rightarrow E$. 
\end{lem}
Let us recall the construction of $\sigma_{E}$ and $\gamma_{E}$
: \\
For a generic point $s$ of $S$, the line $L_{s}$ cuts the hyperplane
section $\psi(\pi^{-1}(E))$ $ $ into a point $p_{s}$. The line
between $p_{s}$ and the vertex of the cone $\psi(\pi^{-1}(E))$ lies
inside this cone and is represented by a point $\gamma_{E}s$ in $E$.
The lines $L_{s}$ and $L_{\gamma_{E}s}$ lies on a plane, that plane
cuts the cubic into a third residual line denoted by $L_{\sigma_{E}s}$.

Let $s$ is a point of $E$ and let $C_{s}$ the incident divisor
parametrizing the lines in $F$ that cut $L_{s}$.
\begin{lem}
\label{lem:equivalence num de la fibre}The fiber $\gamma_{E}^{*}s$
satisfies : $C_{s}=\gamma_{E}^{*}s+E$. We have : $C_{s}^{2}=5$ ,
$C_{s}E=1$ and $E^{2}=-3$. 
\end{lem}
Let $A$ be the Albanese variety of the Fano surface $S$ ; its tangent
space is $H^{0}(S,\Omega_{S})^{*}$. We denote by $\vartheta:S\rightarrow A$
the Albanese map. As $\vartheta$ is an embedding \cite{Clemens},
we consider $S$ as a subvariety of $A$. To the morphisms $\sigma_{E},\,\gamma_{E}$
correspond an involution $\Sigma_{E}:A\rightarrow A$ of $A$ and
a morphism $\Gamma_{E}:A\rightarrow E$ such that: $\vartheta\circ\sigma_{E}=\Sigma_{E}\circ\vartheta$
and $\Gamma_{E}\circ\vartheta=\vartheta\circ\gamma_{E}$.
\begin{lem}
\label{lem:The-differentials-}(\cite{Roulleau1}, Lemma $29$). The
differentials $d\Sigma_{E}$ and $d\Gamma_{E}$ of $\Sigma_{E}$ and
$\Gamma_{E}$ are endomorphisms of $H^{0}(S,\Omega_{S})^{*}$, they
satisfy : \[
I+d\Sigma_{E}+d\Gamma_{E}=0\]
 where I is the identity. The eigenspace of the eigenvalue $1$ of
the involution $d\Sigma_{E}$ is the tangent space $T_{E}$ of the
curve $E\hookrightarrow A$ (translated in $0$). 
\end{lem}
Let us denote by $f$ the projectivization of $d\Sigma_{E}\in GL(H^{0}(S,\Omega_{S})^{*})$
: it is an automorphism of $\mathbb{P}^{4}=\mathbb{P}(H^{0}(S,\Omega_{S})^{*})$.
Let $p_{E}$ be the point of $\mathbb{P}^{4}$ corresponding to the
$1$ dimensional space $T_{E}\subset H^{0}(S,\Omega_{S})^{*}$.
\begin{lem}
The involution $f$ preserves the cubic threefold $F\hookrightarrow\mathbb{P}^{4}$.
\\
The point $p_{E}$ is the vertex of the cone $\psi(\pi^{-1}(E))$.
The hyperplane $\mathbb{P}(Ker(d\Gamma_{E}))$ and $p_{E}$ are the
closed set of fixed points of $f$.\\
Conversely, let $f$ be an involution of $\mathbb{P}^{4}$ acting
on $F$, fixing an isolated point and an hyperplane. The isolated
fixed point is the vertex of a cone on $F$.
\end{lem}
We will use the above Lemma \ref{lem:The-differentials-} as in the
following example:
\begin{example}
\label{example : differentielle de la fibration}Let $x_{1},\dots,x_{5}$
be homogenous coordinates of $\mathbb{P}^{4}$. The point $(1:0:\dots:0)$
is the vertex of a cone on the cubic threefold \[
F=\{x_{1}^{2}x_{2}+G(x_{2},\dots,x_{5})=0\}\]
(where $G$ is a cubic form such that $F$ is smooth). Let $E\hookrightarrow S$
be the elliptic curve parametrizing the lines of that cone. By Lemma
\ref{lem:The-differentials-}, we see that the involution $ $$d\Sigma_{E}$
satisfies:\[
d\Sigma_{E}:(x_{1},x_{2},\dots,x_{5})\rightarrow(x_{1},-x_{2},\dots,-x_{5})\]
and we deduce that $d\Gamma_{E}$ is defined by :\[
d\Gamma_{E}:(x_{1},x_{2},\dots,x_{5})\rightarrow(-2x_{1},0,\dots,0).\]

\end{example}

\subsection{Theta polarization.\label{parag polarisation Theta}}

Let $S$ be a Fano surface, let $A$ be its Albanese variety and let
$\vartheta:S\hookrightarrow A$ be the Albanese map. By \cite{Clemens},
Theorem 13.4, the image $\Theta$ of $S\times S$ under the morphism
$(s_{1},s_{2})\rightarrow\vartheta(s_{1})-\vartheta(s_{2})$ is a
principal polarization of $A$. Let $\tau$ be an automorphism of
$S$ and let $\tau'$ be the automorphism of $A$ such that $\vartheta\circ\tau=\tau'\circ\vartheta$.
Let $(s_{1},s_{2})$ be a point of $S\times S$, then : $\tau'(\vartheta(s_{1})-\vartheta(s_{2}))=\vartheta(\tau(s_{1}))-\vartheta(\tau(s_{2})).$
Thus: 
\begin{lem}
\label{un auto pr=0000E9serve la polarisation}The automorphism $\tau'$
preserves the polarization : $\tau'^{*}\Theta=\Theta$.
\end{lem}
For a variety $X$, we denote by $H^{2}(X,\mathbb{Z})_{f}$ the group
$H^{2}(X,\mathbb{Z})$ modulo torsion. We denote by $\NS(X)=H^{1,1}(X)\cap H^{2}(X,\mathbb{Z})_{f}$
its Néron-Severi group and by $\rho_{X}$ its Picard number. For a
divisor $D$ in $X$, we denote its Chern class by $c_{1}(D)$.\\
The author wishes to thank Bert van Geemen for a useful discussion
on the following Theorem:
\begin{thm}
\label{intersection de deux diviseurs} a) If $D$ and $D'$ are two
divisors of $A$, then: \[
\vartheta^{*}(D)\vartheta^{*}(D')=\int_{A}\frac{1}{3!}\wedge^{3}c_{1}(\Theta)\wedge c_{1}(D)\wedge c_{1}(D').\]
b) The following sequence is exact:\[
0\rightarrow\NS(A)\stackrel{\vartheta^{*}}{\rightarrow}\NS(S)\rightarrow\mathbb{Z}/2\mathbb{Z}\rightarrow0.\]
c) The Néron-Severi group of $S$ is generated by $\vartheta^{*}\NS(A)$
and by the class of an incidence divisor $C_{s}$ ($s\in S$). The
class of $\vartheta^{*}(\Theta)$ is equal to $2C_{s}$. \\
d) We have $\rho_{A}=\rho_{S}\leq25=\dim H^{1}(S,\Omega_{S})$
and $\rho_{S}=1$ for $S$ generic. \end{thm}
\begin{proof}
The morphism $\vartheta$ is an embedding and the homological class
of $\vartheta(S)$ is equal to $\frac{1}{3!}\Theta^{3}$ (\cite{Beauville}
proposition $7$), this proves a).\\
Since $\Theta$ is a polarization, the bilinear symmetric form\[
Q_{\Theta}:H^{2}(A,\mathbb{C})\times H^{2}(A,\mathbb{C})\rightarrow\mathbb{C}\]
 defined by \[
Q_{\Theta}(\eta_{1},\eta_{2})=\int_{A}\frac{1}{3!}\wedge^{3}c_{1}(\Theta)\wedge\eta_{1}\wedge\eta_{2}\]
 is non-degenerate (Hodge-Riemann bilinear relations, section 7 chapter
0 of \cite{Griffiths}). That implies that the morphism \[
\vartheta^{*}:H^{2}(A,\mathbb{C})\rightarrow H^{2}(S,\mathbb{C})\]
 is injective, and since $S$ and $A$ have the same second Betti
number \cite{Gherardelli} (2), it follows that the homomorphism \[
\vartheta_{*}:H_{2}(S,\mathbb{Z})_{f}\rightarrow H_{2}(A,\mathbb{Z})\]
is injective. By \cite{Collino}, $2.3.5.1$, we have the following
exact sequence:\[
H_{2}(S,\mathbb{Z})_{f}\stackrel{\vartheta_{*}}{\rightarrow}H_{2}(A,\mathbb{Z})\rightarrow\mathbb{Z}/2\mathbb{Z}\rightarrow0,\]
thus \[
0\rightarrow H_{2}(S,\mathbb{Z})_{f}\stackrel{\vartheta_{*}}{\rightarrow}H_{2}(A,\mathbb{Z})\rightarrow\mathbb{Z}/2\mathbb{Z}\rightarrow0\]
is exact. By duality, this yields:\[
0\rightarrow H^{2}(A,\mathbb{Z})\stackrel{\vartheta^{*}}{\rightarrow}H^{2}(S,\mathbb{Z})_{f}\rightarrow\mathbb{Z}/2\mathbb{Z}\rightarrow0.\]
As the spaces $H^{1,1}(A)$ and $H^{1,1}(S)$ have dimension $25$
(see \cite{Clemens}), we have : $\vartheta^{*}(H^{1,1}(A))=H^{1,1}(S)$.
This implies that the sequence:\[
0\rightarrow\NS(A)\stackrel{\vartheta^{*}}{\rightarrow}\NS(S)\rightarrow\mathbb{Z}/2\mathbb{Z}\]
 is exact. This sequence is exact on the right also because $\vartheta^{*}\Theta=2C_{s}$
by Lemma 11.27 of \cite{Clemens} ($\Theta$ is a principal polarization,
it is not divisible by $2$, hence the class of $C_{s}$ and $\vartheta^{*}\NS(A)$
generate $\NS(S)$). This proves b).\\
By \cite{Collino}, any Jacobian of an hyperelliptic curve of genus
$5$ is a limit of the Albanese varieties of Fano surfaces endowed
with their principal polarization. By \cite{Mori}, the endomorphism
ring of a Jacobian of a generic hyperelliptic curve is isomorphic
to $\mathbb{Z}$. If the generic Albanese variety of Fano surface
were not simple, then also its limit would be non-simple. This is
a contradiction, hence $\rho_{A}=1$ for a generic Fano surface.
\end{proof}

\section{Fano surfaces with $12$ elliptic curves.}

Let $\lambda\in\mathbb{C}$, $\lambda^{3}\not=1$, the cubic threefold:
\[
F_{\lambda}=\{x_{1}^{3}+x_{2}^{3}+x_{3}^{3}-3\lambda x_{1}x_{2}x_{3}+x_{4}^{3}+x_{5}^{3}=0\}\hookrightarrow\mathbb{P}^{4}\]
is smooth. Let $e_{1},\dots,e_{5}$ be the dual basis of $x_{1},\dots,x_{5}$.
The $12$ points: \[
\mathbb{C}(e_{4}-\beta e_{5}),\,\mathbb{C}(e_{i}-\beta e_{j})\in\mathbb{P}^{4},\,1\leq i<j\leq3,\,\beta^{3}=1\]
 are vertices of cones in $F_{\lambda}$. Let $S_{\lambda}$ be the
Fano surface of $F_{\lambda}$. We denote by $E_{ij}^{\beta}\hookrightarrow S_{\lambda}$
the elliptic curve that parametrizes the lines of the cone of vertex
$\mathbb{C}(e_{i}-\beta e_{j})$ (see Lemma \ref{lem:The-base-of}).
Let $(y_{1}:y_{2}:y_{3})$ be projective coordinates of the plane.
Let $E_{\lambda}\hookrightarrow\mathbb{P}^{2}$ be the elliptic curve:\[
E_{\lambda}=\{y_{1}^{3}+y_{2}^{3}+y_{3}^{3}-3\lambda y_{1}y_{2}y{}_{3}=0\}.\]
with neutral element $(1:-1:0)$. By \cite{Roulleau1}, the $9$ curves
$E_{ij}^{\beta},\,1\leq i<j\leq3,\,\beta^{3}=1$ are disjoint and
isomorphic to $E_{0}$ ; the $3$ curves $E_{45}^{\beta},\,\beta^{3}=1$
are disjoint, isomorphic to $E_{\lambda}$ and : \[
E_{45}^{\beta}E_{ij}^{\gamma}=1,\,\forall\,1\leq i<j\leq3,\,\beta^{3}=\gamma^{3}=1.\]
Conversely, let $S$ be a Fano surfaces with $12$ elliptic curves,
then :
\begin{lem}
\label{lem:12 courbes ell caracterisent}(\cite{Roulleau1}, Parag.
3.3) There is a set of $3$ disjoint elliptic curves on $S$ isomorphic
to $E_{\mu}$ (for some $\mu\in\mathbb{C}$) such that the $9$ remaining
elliptic curves are disjoint and such that the surface $S$ is isomorphic
to $S_{\mu}$.
\end{lem}
Let $Y$ be the surface $Y=E_{\lambda}\times E_{\lambda}$. Let $T_{1}$
and $T_{2}$ be the elliptic curves: \[
\begin{array}{c}
T_{1}=\{x+2y=0/(x,y)\in E_{\lambda}\times E_{\lambda}\}\\
T_{2}=\{2x+y=0/(x,y)\in E_{\lambda}\times E_{\lambda}\}\end{array}\]
on $Y$ and let $\Delta\hookrightarrow Y$ be the diagonal. Any $2$
of the $3$ curves $T_{1},T_{2},\Delta$ meet transversally at the
$9$ points of $3$-torsion of $\Delta$. We denote by $Z$ the blow-up
of $Y$ at these $9$ points.
\begin{prop}
\label{12 courbeselliptiques, divi canonique =00003D somme}The Fano
surface $S_{\lambda}$ is a triple cyclic cover of $Z$ branched along
the proper transform of $\Delta+T_{1}+T_{2}$ in $Z$. \end{prop}
\begin{proof}
Let $\alpha\in\mu_{3}$ be a primitive root. The order $3$ automorphism
\[
f:x\rightarrow(\alpha x_{1}:\alpha x_{2}:\alpha x_{3}:x_{4}:x_{5})\]
acts on $F_{\lambda}$. The automorphism $f$ acts on the Fano surface
of lines of $F_{\lambda}$ by an automorphism denoted by $\tau$.
As we know the action of $f$, we can check immediately that the fixed
locus of $\tau$ is the smooth divisor $E_{45}^{1}+E_{45}^{\alpha}+E_{45}^{\alpha^{2}}$.
The quotient of $S_{\lambda}$ by $\tau$ is a smooth surface $Z'$
with Chern numbers $c_{1}^{2}=-9$ and $c_{2}=9$ and the degree $3$
quotient map $\eta:S_{\lambda}\rightarrow Z'$ is ramified over $E_{45}^{1}+E_{45}^{\alpha}+E_{45}^{\alpha^{2}}$.
\\
For an elliptic curve $E\hookrightarrow S$, we denote by $\gamma_{E}:S\rightarrow E$
the associated fibration (Lemma \ref{lem:construction de la fibration}).
By Lemma \ref{lem:equivalence num de la fibre}, the morphism \[
g=(\gamma_{E_{45}^{\alpha}},\gamma_{E_{45}^{\alpha^{2}}}):S_{\lambda}\rightarrow Y\]
 has degree $3=(C_{s}-E_{45}^{\alpha})(C_{s}-E_{45}^{\alpha^{2}})$.
Let be $E=E_{45}^{\beta}$ (for $\beta^{3}=1$). \\
Let $s$ be a generic point of $S$. By definition (see Lemma \ref{lem:The-base-of}),
the line $L_{s}$ cuts the line $L_{\gamma_{E}s}$. As $L_{\gamma_{E}s}$
is stable by $f$, the line $f(L_{s})=L_{\tau s}$ cuts also the line
$L_{\gamma_{E}s}$, thus, by definition of $\gamma_{E}$, $\gamma_{E}\tau s=\gamma_{E}s$.
That proves that $\gamma_{E}\circ\tau=\gamma_{E}$, and $g\circ\tau=g$.
Hence, by the property of the quotient map, there is a birational
morphism:\[
h:Z'\rightarrow Y\]
such that $g=h\circ\eta$. \\
Let $t$ be the intersection point of $E_{12}^{1}$ and $E_{45}^{1}$
and let $\vartheta:S_{\lambda}\rightarrow A_{\lambda}$ be the Albanese
map such that $\vartheta(t)=0$. It is an embedding and we consider
$S_{\lambda}$ as a subvariety of $A_{\lambda}$. The tangent space
to the curve $E_{45}^{\beta}\hookrightarrow A_{\lambda}$ (translated
to $0$) is $V_{\beta}=\mathbb{C}(\beta e_{4}-\beta^{2}e_{5})$. The
tangent space of $E_{45}^{\alpha}\times E_{45}^{\alpha^{2}}$ is $V_{\alpha}\oplus V_{\alpha^{2}}$.
With the help of Lemma \ref{lem:The-differentials-} and Example \ref{example : differentielle de la fibration},
it is easily checked that the images under $g$ of the curves $E_{45}^{1},E_{45}^{\alpha},E_{45}^{\alpha^{2}}$
are respectively $\Delta,T_{1}$ and $T_{2}$. \\
Moreover, the morphism $g$ has degree $1$ on these $3$ elliptic
curves and contracts the $9$ elliptic curves $E_{ij}^{\beta},\,1\leq i<j\leq3,\,\beta^{3}=1$.
This implies that the image under $g$ of $E_{45}^{1}+E_{45}^{\alpha}+E_{45}^{\alpha^{2}}$
is $T_{1}+T_{2}+\Delta$ and $Z'$ is isomorphic to $Z$. 
\end{proof}
Let $D$ be the proper transform of $\Delta+T_{2}+T_{2}$ in $Z$.
By the above Proposition \ref{12 courbeselliptiques, divi canonique =00003D somme}
and \cite{Barth}, Chap. I, parag. 17 \& 18, the divisor $D$ is divisible
by $3$ in $\NS(Z)$. \\
As $Y$ is an Abelian surface, there exist $3^{4}$ invertible
sheaves $\mathcal{L}$ on $Z$ such that $\mathcal{L}^{\otimes3}=\mathcal{O}_{Z}(D)$.
Let $S(\mathcal{L})\rightarrow Z$ be the degree $3$ cyclic cover
of $Z$ branched over $D$ associated to such $\mathcal{L}$.
\begin{cor}
\label{Corollaire reformulation revetement} The surface $S(\mathcal{L})$
contains $12$ elliptic curves \[
E_{45}^{\beta},\, E_{ij}^{\gamma},1\leq i<j\leq3,\,\beta^{3}=\gamma^{3}=1\]
that have the same configuration as for $S_{\lambda}$, moreover,
the divisor \[
K=\sum_{\beta^{3}=1}2E_{45}^{\beta}+E_{12}^{\beta}+E_{13}^{\beta}+E_{23}^{\beta}\]
is a canonical divisor of $S(\mathcal{L})$.\\
Among these $81$ invertible sheaves, there exists a unique $\mathcal{L}$
such that $S(\mathcal{L})$ is a Fano surface, it is then isomorphic
to $S_{\lambda}$.\end{cor}
\begin{proof}
See \cite{Barth}, Chap. I, parag. 17 \& 18. For the uniqueness of
the invertible sheaf $\mathcal{L}$ : suppose that $S(\mathcal{L})$
and $S(\mathcal{L}')$ are Fano surfaces. By construction, they contain
$12$ elliptic curves, $3$ of them are isomorphic to $E_{\lambda}$
and cut the $9$ others, thus by Lemma \ref{lem:12 courbes ell caracterisent}
$S(\mathcal{L})$ and $S(\mathcal{L}')$ are isomorphic to $S_{\lambda}$,
therefore: $\mathcal{L}=\mathcal{L}'$. \end{proof}
\begin{rem}
The remaining $80$ surfaces $S(\mathcal{L})$ are thus {}``fake''
Fano surfaces and are on different components of the moduli space
of surfaces with $c_{1}^{2}=45$ and $c_{2}=27$.
\end{rem}
Let $\alpha$ be a third primitive root of unity. Let us now study
the Néron-Severi group of $S$.
\begin{prop}
\label{il existe une infinit=0000E9 de NS de rang 25} 1) Suppose
that $E_{\lambda}$ has no complex multiplication. The Néron-Severi
group of $S_{\lambda}$ has rank $12$. The sub-lattice generated
by the elliptic curves and the class of an incidence divisor $C_{s}$
has rank $12$ and discriminant $2.3^{10}$.\\
2) If $E_{\lambda}$ has complex multiplication by a field different
from $\mathbb{Q}(\alpha)$, then the Néron-Severi group of $S_{\lambda}$
has rank $13$.\\
3) If $E_{\lambda}$ has complex multiplication by $\mathbb{Q}(\alpha)$
then the Néron-Severi group of $S_{\lambda}$ has rank $25$.\end{prop}
\begin{proof}
We can easily compute the Picard number of the Abelian variety $E_{0}^{3}\times E_{\lambda}^{2}$.
By \cite{Roulleau1}, the Albanese variety $A$ of $S_{\lambda}$
is isogenous to $E_{0}^{3}\times E_{\lambda}^{2}$, thus their Néron-Severi
groups have same rank and according to the cases 1), 2) and 3), this
rank is $12$, $13$ or $25$. Then, Theorem \ref{intersection de deux diviseurs}
implies that the Picard number of $S_{\lambda}$ is $12$, $13$ or
$25$ respectively.
\end{proof}

\section{The Fano surface of the Fermat cubic.}

\subsection{Elliptic curve configuration of the Fano surface of the Fermat cubic.\label{paraggraphe 441}}

Let $S$ be the Fano surface of the Fermat cubic $F\hookrightarrow\mathbb{P}^{4}=\mathbb{P}(H^{o}(S,\Omega_{S})^{*})$:\[
x_{1}^{3}+x_{2}^{3}+x_{3}^{3}+x_{4}^{3}+x_{5}^{3}=0.\]
Let $e_{1},\dots,e_{5}\in H^{0}(S,\Omega_{S})^{*}$ be the dual basis
of basis of $x_{1},\dots,x_{5}$. Let $\mu_{3}$ be the group of third
roots of unity, let $1\leq i<j\leq5$ and let $\beta\in\mu_{3}$.
The point: \[
p_{ij}^{\beta}=\mathbb{C}(e_{i}-\beta e_{j})\in\mathbb{P}^{4}\]
is the vertex of a cone on the cubic $F$. We denote by $E_{ij}^{\beta}\hookrightarrow S$
the elliptic curve that parametrizes the lines on that cone. The complex
reflection group $G(3,3,5)$ (in the basis $e_{1},\dots,e_{5}$ of
$H^{0}(\Omega_{S})^{*}$) is the group generated by the permutation
matrices and the matrix with diagonal elements $\alpha,\alpha,\alpha,\alpha,\alpha^{2}$
(where $\alpha\in\mu_{3}$ is a primitive root). \\
We recall the following:
\begin{prop}
\label{automor} The Fano surface of the Fermat cubic possesses $30$
smooth curves of genus $1$ numbered :\[
E_{ij}^{\beta},1\leq i<j\leq5,\,\beta\in\mu_{3}.\]
1) Each smooth genus $1$ curve of the Fano surface is isomorphic
to the Fermat plane cubic $\mathbb{E}:=\{x^{3}+y^{3}+z^{3}=0\}$.
\\
2) Let $E_{ij}^{\gamma}$ and $E_{st}^{\beta}$ be two smooth curves
of genus $1$. We have : \[
E_{ij}^{\beta}E_{st}^{\gamma}=\left\{ \begin{array}{cc}
1 & \textrm{if }\{i,j\}\cap\{s,t\}=\emptyset\\
-3 & \textrm{if }E_{ij}^{\beta}=E_{st}^{\gamma}\\
0 & \textrm{else}.\end{array}\right.\]
3) Let $E\hookrightarrow S$ be a smooth curve of genus $1$. The
fibration $\gamma_{E}$ has $20$ sections and contracts $9$ elliptic
curves.\\
4) The automorphism group of $S$ is isomorphic to the complex
reflection $G(3,3,5)$.\end{prop}
\begin{proof}
See \cite{Roulleau1}. For 4), we use the fact that an automorphism
of $F$ must preserve the configuration of the $30$ vertices of cones
and the $100$ lines that contain $3$ such vertices.
\end{proof}

\subsection{The Albanese variety of the Fano surface of the Fermat cubic.}

Let $S$ be the Fano surface of the Fermat cubic $F$. Our main aim
is to compute the full Néron-Severi group of $S$: this will be done
in the next paragraph. We first need to study the Albanese variety
$A$ of $S$.

\subsubsection{Construction of fibrations.}

In order to know the period lattice of the Albanese variety of $S$,
we construct morphisms of the Fano surface onto an elliptic curve
and we study their properties.

Let $\vartheta:S\rightarrow A$ be a fixed Albanese map. It is an
embedding and we consider $S$ as a sub-variety of $A$.\\
 Recall that if $\tau$ is an automorphism of $S$, we denote by
$\tau'\in\Aut(A)$ the unique automorphism such that $\tau'\circ\vartheta=\vartheta\circ\tau$. 

By \cite{Roulleau1}, the reflection group $G(3,3,5)$ is the analytic
representation of the automorphisms $\tau',\,\tau\in\Aut(S)$. The
ring $\mathbb{Z}[G(3,3,5)]\subset\End(H^{0}(\Omega_{S})^{*})$ is
then the analytic representation of a sub-ring of endomorphisms of
the Abelian variety $A$. \\
Let us denote by $\Lambda_{A}^{*}$ the rank $5$ sub-$\mathbb{Z}[\alpha]$-module
of $H^{0}(\Omega_{S})$ generated by the forms: \[
x_{i}-\beta x_{j}\,(i<j,\,\beta\in\mu_{3}).\]
Let $\ell$ be an element of $\Lambda_{A}^{*}$. The endomorphism
of $H^{0}(\Omega_{S})^{*}=H^{0}(S,\Omega_{S})^{*}$ defined by $x\rightarrow\ell(x)(e_{1}-e_{2})$
is an element of $\mathbb{Z}[G(3,3,5)]$. Let us denote by $\Gamma_{\ell}:A\rightarrow\mathbb{E}$
the corresponding morphism of Abelian varieties where $\mathbb{E}\hookrightarrow A$
is the elliptic curve with tangent space $\mathbb{C}(e_{1}-e_{2})$.
We denote by $\gamma_{\ell}:S\rightarrow\mathbb{E}$ the morphism
$\Gamma_{\ell}\circ\vartheta$.

For $1\leq i<j\leq5$ and $\beta\in\mu_{3}$, the space: \[
\mathbb{C}(e_{i}-\beta e_{j})\subset H^{0}(\Omega_{S})^{*}\]
 is the tangent space to the elliptic curve $E_{ij}^{\beta}\hookrightarrow A$
translated to $0$ (Lemma \ref{lem:The-differentials-}).\\
Let $H_{1}(A,\mathbb{Z})\subset H^{0}(\Omega_{S})^{*}$ be the
period lattice of $A$. The elliptic curve $\mathbb{E}$ has complex
multiplication by the principal ideal domain $\mathbb{Z}[\alpha]$.
There exists $c\in\mathbb{C}^{*}$ such that : \[
H_{1}(A,\mathbb{Z})\cap\mathbb{C}(e_{1}-e_{2})=\mathbb{Z}[\alpha]c(e_{1}-e_{2}).\]
Up to the basis change of $e_{1},\dots,e_{5}$ by $ce_{1},\dots,ce_{5}$,
we may suppose $c=1$. Since $G(3,3,5)$ acts transitively on the
$30$ spaces $\mathbb{C}(e_{i}-\beta e_{j})$, we have: \[
H_{1}(A,\mathbb{Z})\cap\mathbb{C}(e_{i}-\beta e_{j})=\mathbb{Z}[\alpha](e_{i}-\beta e_{j}).\]
We define the Hermitian product of two forms $\ell,\ell'\in\Lambda_{A}^{*}$
by : \[
\left\langle \ell,\ell'\right\rangle :=\sum_{k=1}^{k=5}\ell(e_{k})\overline{\ell'(e_{k})},\]
 and the norm of $\ell$ by: $\left\Vert \ell\right\Vert =\sqrt{\left\langle \ell,\ell\right\rangle }$.
Let $C_{s}$ be an incidence divisor.
\begin{thm}
\label{toutes les fibrations}Let $\ell$ be a non zero element of
$\Lambda_{A}^{*}$ and let $F_{\ell}$ be a fibre of $\gamma_{\ell}$.
\\
1) The intersection number of $F_{\ell}$ and $E_{ij}^{\beta}\hookrightarrow S$
is equal to: \[
E_{ij}^{\beta}F_{\ell}=|\ell(e_{i}-\beta e_{j})|^{2}.\]
 2) We have: $F_{\ell}C_{s}=2\left\Vert \ell\right\Vert ^{2}$ and
the fibre $F_{\ell}$ has genus:\[
g(F_{\ell})=1+3\left\Vert \ell\right\Vert ^{2}.\]
3) Let $\ell$ and $\ell'$ be two linearly independent elements of
$\Lambda_{A}^{*}\subset H^{0}(\Omega_{S})$. The morphism $\tau_{\ell,\ell'}=(\gamma_{\ell},\gamma_{\ell'}):S\rightarrow\mathbb{E}\times\mathbb{E}$
has degree equal to $F_{\ell}F_{\ell'}$ and : \[
F_{\ell}F_{\ell'}=\left\Vert \ell\right\Vert ^{2}\left\Vert \ell'\right\Vert ^{2}-\left\langle \ell,\ell'\right\rangle \left\langle \ell',\ell\right\rangle .\]
\end{thm}
\begin{rem}
The known intersection numbers $F_{\ell}E_{ij}^{\beta}$ and $F_{\ell}C_{s}$
enables us to write the numerical equivalence class of the fibre $F_{\ell}$
in the $\mathbb{Z}$-basis given in Theorem \ref{neronseveri} below. 
\end{rem}
Let us prove Theorem \ref{toutes les fibrations}. The part 1) is
a trick:\\
For $1\leq i<j\leq5$ and $\beta\in\mu_{3}$, we can interpret
geometrically the intersection number $E_{ij}^{\beta}F_{\ell}$ as
the degree of the restriction of $\gamma_{\ell}$ to $E_{ij}^{\beta}\hookrightarrow S$.
It is also the degree of the restriction of $\Gamma_{\ell}$ to $E_{ij}^{\beta}\hookrightarrow A$.
As this restriction is the multiplication map by $\ell(e_{i}-\beta e_{j})$,
the degree of the morphism $\Gamma_{\ell}$ on $E_{ij}^{\beta}$ is
equal to $|\ell(e_{i}-\beta e_{j})|^{2}$. Thus $F_{\ell}E_{ij}^{\beta}=|\ell(e_{i}-\beta e_{j})|^{2}$.

Let us study the genus of $F_{\ell}$:
\begin{lem}
\label{genre de la fibre f ell}The fibre of $F_{\ell}$ has genus
$1+3\left\Vert \ell\right\Vert ^{2}$ and $C_{s}F_{\ell}=2\left\Vert \ell\right\Vert ^{2}$
(where $s$ is a point of $S$ and $C_{s}$ the incidence divisor).\end{lem}
\begin{proof}
Let $\Sigma$ be the sum of the $30$ elliptic curves on $S$. We
have: \[
\Sigma F_{\ell}=\sum_{i,j,\beta}F_{\ell}E_{ij}^{\beta}=\sum_{i,j,\beta}|\ell(e_{i}-\beta e_{j})|^{2}=12\left\Vert \ell\right\Vert ^{2}.\]
As $F_{\ell}$ is a fibre, we have $F_{\ell}^{2}=0$ and since $\Sigma$
is twice a canonical divisor \cite{Clemens}, we deduce that $F_{\ell}$
has genus $1+\frac{1}{2}(0+\frac{1}{2}\Sigma F_{\ell})=1+3\left\Vert \ell\right\Vert ^{2}$.
\\
The divisor $3C_{s}$ is numerically equivalent to a canonical
divisor \cite{Clemens}. Thus: $C_{s}F_{\ell}=2\left\Vert \ell\right\Vert ^{2}$.
\end{proof}
We identify the Chern class of a divisor of the Abelian variety $A$
with an alternating form on the tangent space $H^{0}(\Omega_{S})^{*}$
of $A$ (\cite{Birkenhake}, Theorem 2.12). Let $\Theta$ be the principal
polarization defined in paragraph \ref{parag polarisation Theta}.
\begin{lem}
\label{la forme de chern de pol fermat}The Chern Class of $\Theta$
is equal to: \[
a\frac{i}{\sqrt{3}}\sum_{j=1}^{5}dx_{j}\wedge d\bar{x}_{j}\]
where $a$ is a scalar and $i^{2}=-1$. \end{lem}
\begin{proof}
Let $H$ be the matrix (in the basis $e_{1},\dots,e_{5}$) of the
Hermitian form associated to $c_{1}(\Theta)$ (see \cite{Birkenhake},
Lemma 2.17). The automorphism $\tau'$ induced by $\tau\in\Aut(S)$
preserves the polarization $\Theta$ (Lemma \ref{un auto pr=0000E9serve la polarisation}).
This implies that for all $M=(m_{jk})_{1\leq j,k\leq5}\in G(3,3,5)$,
we have :\[
^{t}MH\bar{M}=H\]
(where $\bar{M}$ is the matrix $\bar{M}=(\bar{m}_{jk})_{1\leq j,k\leq5}$)
and this proves that \[
H=\frac{2}{\sqrt{3}}aI_{5}\]
 where $I_{5}$ is the identity matrix and $a\in\mathbb{C}$. Hence:
$c_{1}(\Theta)=a\frac{i}{\sqrt{3}}\sum_{j=1}^{5}dx_{j}\wedge d\bar{x}_{j}$.
\end{proof}
Since $H_{1}(A,\mathbb{Z})\cap\mathbb{C}(e_{1}-e_{2})=\mathbb{Z}[\alpha](e_{1}-e_{2})$,
the Néron-Severi group of the elliptic curve $\mathbb{E}$ is the
$\mathbb{Z}$-module generated by \[
\eta=\frac{i}{\sqrt{3}}dz\wedge d\bar{z}\]
where $z$ is the coordinate on the space $\mathbb{C}(e_{1}-e_{2})$.\\
Let $\ell=a_{1}x_{1}+\dots+a_{5}x_{5}$ be an element of $\Lambda_{A}^{*}$.
The pull back of the form $\eta$ by the morphism $\Gamma_{\ell}:A\rightarrow\mathbb{E}$
is: \[
\Gamma_{\ell}^{*}\eta=\frac{i}{\sqrt{3}}d\ell\wedge d\overline{\ell}.\]
The form $\Gamma_{\ell}^{*}\eta$ is the Chern class of the divisor
$\Gamma_{\ell}^{*}0$ and $\gamma_{\ell}^{*}\eta=\vartheta^{*}\Gamma_{\ell}^{*}\eta$
is the Chern class of the divisor $F_{\ell}$.
\begin{lem}
\label{la forme Q}Let $\ell$ and $\ell'$ be two elements of $\Lambda_{A}^{*}$,
then: \[
F_{\ell}F_{\ell'}=\left\Vert \ell\right\Vert ^{2}\left\Vert \ell'\right\Vert ^{2}-\left\langle \ell,\ell'\right\rangle \left\langle \ell',\ell\right\rangle \]
and $c_{1}(\Theta)=\frac{i}{\sqrt{3}}\sum_{i=1}^{i=5}dx_{i}\wedge d\bar{x}_{i}$.\end{lem}
\begin{proof}
By the Theorem \ref{intersection de deux diviseurs}, $\vartheta^{*}c_{1}(\Theta)$
is the Chern class of the divisor $2C_{s}$ ($s\in S$) and:\[
2C_{s}F_{\ell}=\vartheta^{*}c_{1}(\Theta)\vartheta^{*}\Gamma_{\ell}^{*}\eta=\int_{A}\frac{1}{3!}\wedge^{4}c_{1}(\Theta)\wedge\Gamma_{\ell}^{*}\eta\]
hence: \[
2C_{s}F_{\ell}=(\frac{i}{\sqrt{3}})^{5}\int_{A}(\sum a_{j}dx_{j})\wedge(\sum\bar{a}_{j}d\bar{x}_{j})\wedge4a^{4}\sum_{1\leq k\leq5}(\wedge_{j\not=k}(dx_{j}\wedge d\bar{x}_{j}))\]
and: \[
2C_{s}F_{\ell}=(\frac{4}{a}\sum_{k=1}^{k=5}a_{k}\bar{a}_{k})\frac{1}{5!}\int_{A}\wedge^{5}c_{1}(\Theta).\]
Since $\Theta$ is a principal polarization, we have $\frac{1}{5!}\int_{A}\wedge^{5}c_{1}(\Theta)=1$,
hence: $2C_{s}F_{\ell}=\frac{4}{a}\left\Vert \ell\right\Vert ^{2}$.
We have seen in Lemma \ref{genre de la fibre f ell} that $C_{s}F_{\ell}=2\left\Vert \ell\right\Vert ^{2}$.
Thus we deduce that $a=1$.\\
By Theorem \ref{intersection de deux diviseurs}, for $\ell=a_{1}x_{1}+\dots+a_{5}x_{5}$
and $\ell'=b_{1}x_{1}+\dots+b_{5}x_{5}\in\Lambda_{A}^{*}$, \[
F_{\ell}F_{\ell'}=\int_{A}\frac{1}{3!}\wedge^{3}c_{1}(\Theta)\wedge\Gamma_{\ell}^{*}\eta\wedge\Gamma_{\ell'}^{*}\eta.\]
 Since \[
\frac{1}{3!}(\frac{i}{\sqrt{3}})^{2}d\ell\wedge d\overline{\ell}\wedge d\ell'\wedge d\overline{\ell'}\wedge(\wedge^{3}c_{1}(\Theta))=(\sum_{k\not=j}a_{k}\bar{a}_{k}b_{j}\bar{b}_{j}-a_{k}\bar{a}_{j}b_{j}\bar{b}_{k})\frac{1}{5!}\wedge^{5}c_{1}(\Theta),\]
 the result follows.
\end{proof}
Let $\ell$ and $\ell'$ be two linearly independent elements of $\Lambda_{A}^{*}$.
The degree of the morphism $\tau_{\ell,\ell'}=(\gamma_{\ell},\gamma_{\ell'})$
is equal to $F_{\ell}F_{\ell'}$ because $\tau_{\ell,\ell'}^{*}(\mathbb{E}\times\{0\})=F_{\ell'}\in\NS(S)$,
$\tau_{\ell,\ell'}^{*}(\{0\}\times\mathbb{E})=F_{\ell}\in\NS(S)$
and the intersection number of the divisors $\{0\}\times\mathbb{E}$
and $\mathbb{E}\times\{0\}$ is equal to $1$. 

This completes the proof of Theorem \ref{toutes les fibrations}.
$\Box$

\subsubsection{Period lattice of $A$.}

We compute here the period lattice of the Albanese variety $A$ in
the basis $e_{1},\dots,e_{5}$.
\begin{thm}
\label{le reseau de A}The lattice $H_{1}(A,\mathbb{Z})$ is equal
to:\[
\begin{array}{cc}
\mathbb{Z}[\alpha](e_{1}-e_{5}) & +\mathbb{Z}[\alpha](e_{2}-e_{5})+\mathbb{Z}[\alpha](e_{3}-e_{5})+\mathbb{Z}[\alpha](e_{4}-e_{5})\\
 & +\frac{1+\alpha}{1-\alpha}\mathbb{Z}[3\alpha](\alpha^{2}e_{1}+\alpha^{2}e_{2}+\alpha e_{3}+\alpha e_{4}+e_{5}).\end{array}\]
The variety $A$ is isomorphic to $\mathbb{E}^{4}\times\mathbb{E}'$
where $\mathbb{E}=\mathbb{C}/\mathbb{Z}[\alpha]$ and $\mathbb{E}'=\mathbb{C}/\mathbb{Z}[3\alpha]$.\\
The image of the morphism $\vartheta^{*}:\NS(A)\rightarrow\NS(S)$
is the sub-lattice of rank $25$ and discriminant $2^{2}3^{18}$ generated
by the divisors : \[
F_{x_{i}-\beta^{2}x_{j}}=C_{s}-E_{ij}^{\beta},\,1\leq i<j\leq5,\,\beta\in\mu_{3}\textrm{ and }\sum_{i<j}E_{ij}^{1}.\]
\end{thm}
\begin{proof}
The group $G(3,3,5)$ acts on $H_{1}(A,\mathbb{Z})$ and : \[
H_{1}(A,\mathbb{Z})\cap\mathbb{C}(e_{i}-\beta e_{j})=\mathbb{Z}[\alpha](e_{i}-\beta e_{j}),\]
hence $H_{1}(A,\mathbb{Z})$ contains the lattice $\Lambda_{0}=\sum_{i\leq j,\beta\in\mu_{3}}\mathbb{Z}[\alpha](e_{i}-\beta e_{j})$.
\\
For $1\leq i<j\leq5,\,\beta\in\mu_{3}$, the differential of $\Gamma_{x_{i}-\beta x_{j}}$
is the morphism $x\rightarrow(x_{i}-\beta x_{j})(e_{1}-e_{2})$. Thus:
\[
\forall\lambda=(\lambda_{1},\dots,\lambda_{5})\in H_{1}(A,\mathbb{Z}),\,\lambda_{i}-\beta\lambda_{j}\in\mathbb{Z}[\alpha].\]
 Let us define \[
\Lambda=\{x=(x_{1},.\dots,x)\in\mathbb{C}^{5}/x_{i}-\beta x_{j}\in\mathbb{Z}[\alpha],1\leq i<j\leq5,\beta\in\mu_{3}\}.\]
 This lattice $\Lambda$ contains $H_{1}(A,\mathbb{Z})$ and is equal
to\[
\mathbb{Z}[\alpha]e_{1}\oplus\dots\oplus\mathbb{Z}[\alpha]e_{4}\oplus\frac{1}{\alpha-1}\mathbb{Z}[\alpha]w,\]
where $w=e_{1}+\dots+e_{5}$. Let $\phi:\Lambda\rightarrow\Lambda/\Lambda_{0}$
be the quotient map. The group $\Lambda/\Lambda_{0}$ is isomorphic
to $(\mathbb{Z}/3\mathbb{Z})^{2}$ and contains $6$ sub-groups. The
reciprocal images of these groups are the lattices\[
\begin{array}{cc}
\begin{array}{ccc}
 & \Lambda_{0}= & \phi^{-1}(0)\\
 & \Lambda_{1}= & \Lambda_{0}+\frac{1}{\alpha-1}\mathbb{Z}w\\
 & \Lambda_{\alpha}= & \Lambda_{0}+\frac{\alpha}{\alpha-1}\mathbb{Z}w\end{array} & \begin{array}{ccc}
 & \Lambda_{\alpha^{2}}= & \Lambda_{0}+\frac{\alpha^{2}}{\alpha-1}\mathbb{Z}w\\
 & \Lambda_{\alpha-1}= & \Lambda_{0}+\mathbb{Z}w\\
 & \Lambda= & \Lambda_{0}+\frac{1}{\alpha-1}\mathbb{Z}[\alpha]w.\end{array}\end{array}\]
These are the $6$ lattices $\Lambda'$ which verify $\Lambda_{0}\subset\Lambda'\subset\Lambda$,
thus the lattice $H_{1}(A,\mathbb{Z})$ is equal to one of these.

Let $\omega$ be the alternating form $\omega=\frac{i}{\sqrt{3}}\sum_{k=1}^{k=5}dx_{k}\wedge d\bar{x}_{k}$
(see Lemma \ref{la forme Q}). We have \[
\frac{1}{\alpha-1}w,\frac{\alpha}{\alpha-1}w\in\Lambda.\]
 However \[
\omega(\frac{1}{\alpha-1}w,\frac{\alpha}{\alpha-1}w)=-\frac{5}{3}\]
 is not an integer, hence $\Lambda$ is different from $H_{1}(A,\mathbb{Z})$.\\
The Pfaffian of $c_{1}(\Theta)$ relative to $H_{1}(A,\mathbb{Z})$
is equal to $1$ because $\Theta$ is a principal polarization. The
Pfaffian of $\omega$ relative to the lattice $\Lambda_{0}$ is equal
to $9$, hence $\Lambda_{0}$ is different from $H_{1}(A,\mathbb{Z})$.\\
We have $\Lambda_{1-\alpha}=\oplus\mathbb{Z}[\alpha]e_{i}$ and
the principally polarized Abelian variety $(\mathbb{C}^{5}/\Lambda_{1-\alpha},\omega)$
is isomorphic to a product of Jacobians. Since $(A,c_{1}(\Theta))$
cannot be isomorphic to a product of Jacobians (\cite{Clemens}, 0.12),
$H_{1}(A,\mathbb{Z})\not=\Lambda_{1-\alpha}$.\\
The lattice $\Lambda_{\alpha^{i}}$ is equal to:\[
\begin{array}{cc}
\mathbb{Z}[\alpha](e_{1}-e_{5}) & +\mathbb{Z}[\alpha](e_{2}-e_{5})+\mathbb{Z}[\alpha](e_{3}-e_{5})+\mathbb{Z}[\alpha](e_{4}-e_{5})\\
 & +\frac{\alpha^{i}}{1-\alpha}\mathbb{Z}[3\alpha](\alpha^{2}e_{1}+\alpha^{2}e_{2}+\alpha e_{3}+\alpha e_{4}+e_{5}).\end{array}\]
The lattices $\Lambda_{1}$ and $\Lambda_{\alpha}$ depend upon the
choice of $\alpha$ such that $\alpha^{2}+\alpha+1=0$, hence the
lattice $H_{1}(A,\mathbb{Z})$ is equal to $\Lambda_{\alpha^{2}}$.

Let be \[
\begin{array}{c}
u_{1}=e_{1}-e_{2},\, u_{2}=e_{2}-e_{3},\, u_{3}=e_{3}-e_{4},\, u_{4}=e_{4}-e_{5},\\
u_{5}=\frac{\alpha^{2}}{1-\alpha}(\alpha^{2}e_{1}+\alpha^{2}e_{2}+\alpha e_{3}+\alpha e_{4}+e_{5})\end{array}\]
The Hermitian form $H'=\frac{2}{\sqrt{3}}I_{5}$ in the basis $u_{1},\ldots,u_{5}$
defines a principal polarization of $A$. Let $\End^{s}(A)$ be the
group of symmetrical morphisms for the Rosati involution associated
to $H'$. An endomorphism of $A$ can be represented by a size $5$
matrix $M$ in the basis $u_{1},\ldots,u_{5}$. The symmetrical endomorphisms
satisfy $^{t}MH'=H'\bar{M}$ i.e. $^{t}A=\bar{A}$. As we know $H_{1}(A,\mathbb{Z})$,
we can easily compute a basis $\mathcal{B}$ of $\End^{s}(A)$.\\
By \cite{Birkenhake}, Proposition $5.2.1$ and Remark 5.2.2.,
the map:\[
\begin{array}{ccc}
\phi_{H'}:\End^{s}(A) & \rightarrow & \NS(A)\\
M & \mapsto & \Im m(\cdot^{t}MH'\bar{\cdot})\end{array}\]
is an isomorphism of groups. We obtain the base of the Néron-Severi
group of $A$ by taking the image by $\phi_{H'}$ of the base $\mathcal{B}$.
Then we get the image of the morphism $\vartheta^{*}:\NS(A)\rightarrow\NS(S)$. \end{proof}
\begin{rem}
By Theorem \ref{le reseau de A}, the variety $A$ is biregular to
$E^{4}\times E'$, but by \cite{Clemens} $(0.12)$, this cannot be
an isomorphism of principally polarized Abelian varieties.
\end{rem}

\subsubsection{Study of some fibrations, remarks.}

Let $X$ be a smooth surface, $C$ a smooth curve, $\gamma:X\rightarrow C$
a fibration with connected fibres. A point of $X$ is called a critical
point of $\gamma$ if it is a zero of the differential: \[
d\gamma:T_{X}\rightarrow\gamma^{*}T_{C}.\]
 A fibre of $\gamma$ is singular at a point if and only if this point
is a critical point (\cite{Barth} Chapter III, section 8). \\
Let us suppose that $C$ is an elliptic curve. The critical points
of $\gamma$ are then the zeros of the form $\gamma^{*}\omega\in H^{o}(X,\Omega_{X})$
where $\omega$ is a generator of the trivial sheaf $\Omega_{C}$.
\\
Let us assume that the cotangent sheaf of $X$ is generated by
global sections. There are morphisms:\[
\begin{array}{ccc}
\mathbb{P}(T_{X}) & \stackrel{\psi}{\rightarrow} & \mathbb{P}(H^{0}(X,\Omega_{X})^{*})\\
\pi\downarrow\\
X\end{array}\]
where $\pi$ is the natural projection and the morphism $\psi$, called
the cotangent map \cite{Roulleau}, is defined by $\pi_{*}\psi^{*}\mathcal{O}(1)=\Omega_{X}$.
\\
A point $x$ of $X$ is a critical point of $\gamma$ if and only
if the line $L_{x}=\psi(\pi^{-1}(x))$ lies in the hyperplane \[
\{\gamma^{*}\omega=0\}\hookrightarrow\mathbb{P}(H^{o}(X,\Omega_{X})^{*}).\]

The initial motivation for studying Fano surfaces is the fact that
for those surfaces, the cotangent map is known : it is the projection
map of the universal family of lines. The example of Fano surfaces
gives an illustration that the knowledge of the image of the cotangent
map is powerful for the study of a surface.\\
Let $S$ be the Fano surface of the Fermat cubic $F$. 
\begin{notation}
\label{notation de somme de 3 courbes}For $1\leq i<j\leq5$, we define
$B_{ij}=B_{ji}=\sum_{\beta\in\mu_{3}}E_{ij}^{\beta}$.
\end{notation}
$\blacksquare$ Let $1\leq i\leq5$ and $j<r<s<t$ be such that $\{i,j,r,s,t\}=\{1,2,3,4,5\}$.
We define $\ell_{i}=(1-\alpha)x_{i}\in\Lambda_{A}^{*}$.
\begin{cor}
\label{exemple de fibration}The fibration $\gamma_{\ell_{i}}:S\rightarrow\mathbb{E}$
is stable, has connected fibres of genus $10$ and its only singular
fibres are: \[
B_{jr}+B_{st},\, B_{js}+B_{rt},\, B_{jt}+B_{rs}.\]
The $27$ intersection points of the curves $E_{jr}^{\beta}$ and
$E_{st}^{\tau}$ ($\beta,\gamma\in\mu_{3}$) constitute the set of
critical points of this fibration. \end{cor}
\begin{proof}
By Theorem \ref{toutes les fibrations} 1), a fibre of $\gamma_{\ell_{i}}$
has genus $1+3|1-\alpha|^{2}=10$. \\
Let $\beta\in\mu_{3},\, h,k\in\{j,r,s,t\},\, h<k$. The form $\ell_{i}$
is zero on the space $\mathbb{C}(e_{h}-\beta e_{k})$, hence $E_{hk}^{\beta}$
is contracted to a point and is a component of the fibre of $\gamma_{\ell_{i}}$.\\
The divisor $D_{1}=B_{jr}+B_{st}$ is connected, satisfies $(B_{jr}+B_{st})^{2}=0$
and has genus $10$. Its irreducible components are contracted by
$\gamma_{\ell_{i}}$. Hence, it is a fibre and $\gamma_{\ell_{i}}$
has connected fibres. Likewise, the divisors $D_{2}=B_{js}+B_{rt}$
and $D_{3}=B_{jt}+B_{rs}$ are fibres of $\gamma_{\ell_{i}}$. \\
The $27$ lines inside the intersection of the Fermat cubic and
the hyperplane $\{\ell_{i}=0\}$ correspond to the $27$ intersection
points of the curves $E_{hk}^{\beta}$ and $E_{lm}^{\gamma}$ such
that : $h<k,\, l<m$ and $\{i,h,k,l,m\}=\{1,2,3,4,5\}$. These $27$
critical points lie in the fibres $D_{1},D_{2},D_{3}$. These $3$
fibres are thus the only singular fibres of $\gamma_{\ell_{i}}$.
\\
The singularities of $D_{1},D_{2}$ and $D_{3}$ are double ordinary
and the surface possesses no rational curve, the fibration is thus
stable.
\end{proof}
$\blacksquare$ Let $(a_{1},\dots,a_{5})\in\mu_{3}^{5}$ be such that
$a_{1}...a_{5}=1$ and let \[
\ell=(1-\alpha)(a_{1}x_{1}+\dots+a_{5}x_{5})\in\Lambda_{A}^{*}.\]

\begin{cor}
\label{diviseur somme des 10 courbes} The divisor \[
D=\sum_{1\leq i<j\leq5}E_{ij}^{a_{i}/a_{j}}\]
is a singular fibre of the Stein factorization of $\gamma_{\ell}$.\end{cor}
\begin{proof}
The connected divisor $D$ satisfies $D^{2}=0$, has genus $16$ and
by Theorem \ref{toutes les fibrations}, the irreducible components
of $D$ are contracted by $\gamma_{\ell}$.\\
Let $w\in H^{o}(\Omega_{S})^{*}$ be : $w=e_{1}+\dots+e_{5}$.
We have \[
H^{1}(A,\mathbb{Z})\cap\mathbb{C}w=\frac{\alpha^{2}}{1-\alpha}\mathbb{Z}[3\alpha]w.\]
The morphism $x\rightarrow(a_{1}x_{1}+\dots+a_{5}x_{5})w\in End(H^{o}(\Omega_{S})^{*})$
is an element of $\mathbb{Z}[G(3,3,5)]$. It is the differential of
a morphism $\Gamma'_{\ell}:A\rightarrow\mathbb{E}'$ where $\mathbb{E}'=(\mathbb{C}/\frac{\alpha^{2}}{1-\alpha}\mathbb{Z}[3\alpha])w$.
\\
The morphism $\Gamma_{\ell}$ has a factorization by $\Gamma_{\ell}'$
and a degree $3$ isogeny between $\mathbb{E}'$ and $\mathbb{E}$.
The divisor $D$ is a connected fibre of the morphism $\vartheta\circ\Gamma_{\ell}'$,
the Stein factorization of $\gamma_{\ell}$.
\end{proof}
$\blacksquare$ The curve $E_{ij}^{\beta^{2}}$ is the closed set
of critical points of the fibration \[
\gamma_{(1-\alpha)(x_{i}+\beta x_{j})}.\]
This fibration has only one singular fiber and this fiber is not reduced.

$\blacksquare$ We can construct an infinite number of fibrations
with $9$ sections and which contract $9$ elliptic curves. Let us
take $a\in\mathbb{Z}[\alpha]$ and \[
\ell=x_{1}-(1+(1-\alpha)a)x_{2}\in\Lambda_{A}^{*}.\]

\begin{cor}
The $9$ curves $E_{13}^{\beta},E_{14}^{\beta}$ and $E_{15}^{\beta}$
($\beta\in\mu_{3}$) are sections of $\gamma_{\ell}$. The $9$ curves
$E_{34}^{\beta},E_{35}^{\beta}$ and $E_{45}^{\beta}$ ($\beta\in\mu_{3}$)
are contracted.\end{cor}
\begin{proof}
This follows from Theorem \ref{toutes les fibrations} and the fact
that \[
\begin{array}{cc}
|\ell(e_{1}-\beta e_{3})|=|\ell(e_{1}-\beta e_{4})|=|\ell(e_{1}-\beta e_{5})|=1, & (\beta\in\mu_{3})\\
|\ell(e_{3}-\beta e_{4})|=|\ell(e_{3}-\beta e_{5})|=|\ell(e_{4}-\beta e_{5})|=0.\end{array}\]
As $F_{\ell}E_{12}^{1}=1$, the fibration $\gamma_{\ell}$ has connected
fibres.
\end{proof}

\subsection{The Néron-Severi group of the Fano surface of the Fermat cubic.\label{paragraphe neron severi}\protect \\
}

Let $S$ be the Fano surface of the Fermat cubic and $\NS(S)$ the
Néron-Severi group of $S$. 
\begin{thm}
\label{neronseveri} The Néron-Severi group of $S$ has rank $25=\dim H^{1}(S,\Omega_{S})$.
The $30$ elliptic curves generate an index $3$ sub-lattice of $\NS(S)$.
\\
The group $\NS(S)$ is generated by these $30$ curves and the
class of an incidence divisor $C_{s}$ ($s\in S$), it has discriminant
$3^{18}$.\\
The relations between the $30$ elliptic curves in $\NS(S)$ are
generated by the relations :\[
B_{jr}+B_{st}=B_{js}+B_{rt}=B_{jt}+B_{rs},\]
for indices such that $1\leq j<r<s<t\leq5$. \end{thm}
\begin{proof}
By Theorem \ref{automor}, we know the intersection matrix $\mathcal{I}$
of the $30$ elliptic curves. As we can verify, the matrix $\mathcal{I}$
has rank $25$. The intersection matrix of the $25$ elliptic curves
different to the $5$ curves \[
E_{13}^{\alpha},E_{15}^{\alpha},E_{24}^{\alpha},E_{34}^{\alpha},E_{45}^{\alpha}\]
has determinant equal to $3^{20}$ and these $25$ curves form a $\mathbb{Z}$-basis
of the lattice generated by the $30$ elliptic curves.\\
By Theorem \ref{le reseau de A}, the image of the morphism\[
\NS(A)\stackrel{\vartheta^{*}}{\rightarrow}\NS(S)\]
is a lattice of discriminant $2^{2}3^{18}$ generated by the class
of $C_{s}-E_{ij}^{\beta}$ and by $\sum_{1\leq i<j\leq5}E_{ij}^{1}$.
Theorem \ref{intersection de deux diviseurs} implies that $\NS(S)$
is generated by these classes and the class of an incidence divisor
$C_{s}$. This lattice is also generated by the classes of the $30$
elliptic curves and $C_{s}$.
\end{proof}
We remark that Corollary \ref{exemple de fibration} gives a geometric
interpretation of the numerical equivalence relations of Theorem \ref{neronseveri}.

Xavier Roulleau,\\
Graduate School of Mathematical Sciences, \\
University of Tokyo, \\
3-8-1 Komaba, Meguro, Tokyo, \\
153-8914 Japan
\end{document}